\documentclass[11pt,reqno]{amsart}

\usepackage[T1]{fontenc}

\usepackage{amsmath}								
\usepackage{amssymb}
\usepackage{amsthm}
\usepackage{amscd}
\usepackage{amsfonts}
\usepackage{stmaryrd}
\usepackage{algorithm, algorithmic}

\usepackage{caption}
\usepackage{subcaption}
\usepackage{euler}

\usepackage{extarrows}

\usepackage[colorlinks, linktocpage, citecolor = red, linkcolor = blue]{hyperref}
\usepackage{color}

\usepackage{tikz}									
\usetikzlibrary{matrix}
\usetikzlibrary{patterns}
\usetikzlibrary{matrix}
\usetikzlibrary{positioning}
\usetikzlibrary{decorations.pathmorphing}
\usetikzlibrary{cd}

\usepackage{fullpage}
\usepackage[shortlabels]{enumitem}

\newtheorem{theorem}{Theorem}[section]
\newtheorem{lemma}[theorem]{Lemma}
\newtheorem{proposition}[theorem]{Proposition}
\newtheorem{corollary}[theorem]{Corollary}

\theoremstyle{definition}
								
\newtheorem{definition}[theorem]{Definition}

\newtheorem{example}[theorem]{Example}

\newtheorem{remark}[theorem]{Remark}

\newcommand{\ZZ}{\mathbb{Z}}

\newcommand{\PP}{\mathbb{P}}

\newcommand{\RR}{\mathbb{R}}
\newcommand{\TT}{\mathbb{T}}

\newcommand{\FF}{\mathbb{F}}

\newcommand{\M}{\mathcal{M}}
\newcommand{\ppp}{\mathcal{P}}

\newcommand{\B}{\mathcal{B}}

\DeclareMathOperator{\Gr}{Gr}
\DeclareMathOperator{\Dr}{Dr}
\DeclareMathOperator{\conv}{conv}

\DeclareMathOperator{\Hom}{Hom}

\DeclareMathOperator{\trop}{trop}
\DeclareMathOperator{\val}{val}

\DeclareMathOperator{\Linearity}{S} 
\DeclareMathOperator{\Lineality}{s}

\title[Matroids and their Dressians]{Matroids and their Dressians}

\author{Madeline Brandt}
\address{Department of Mathematics, Brown University, 151 Thayer St, Providence, RI, 02912}
\email{\href{mailto:madeline_brandt@brown.edu}{madeline\_brandt@brown.edu}}

\author{David E Speyer}
\address{Department of Mathematics, University of Michigan, 2844 East Hall, Ann Arbor, MI, 48109}
\email{\href{mailto:speyer@umich.edu}{speyer@umich.edu}}

\begin{document}

\begin{abstract} 
We study Dressians of matroids using the initial matroids of Dress and Wenzel.
These correspond to cells in regular matroid subdivisions of matroid polytopes.
An efficient algorithm for computing Dressians is presented, and its
implementation is applied to a range of interesting matroids.
We give counterexamples to a few plausible statements about matroid subdivisions.
\end{abstract}

\maketitle

\setcounter{tocdepth}{1}
\tableofcontents


\section*{Introduction}

Let $K$ be an algebraically closed field with a non-trivial valuation, valuation ring $R$, and residue field $k$.
Consider a collection of vectors $v_1, \ldots, v_n \in R^d$ spanning $K^d$. These vectors give a rank $d$ matroid $\M$ on $n$ elements, whose bases are given by the bases of $K^d$ coming from the $v_1, \ldots, v_n$.
If we pass these vectors to the residue field $k$, their images will generate a matroid $\M'$, called an initial matroid of $\M$, which is a special kind of weak image of $\M$. 
One can also expand these ideas to non-realizable matroids, and the Dressian is the tropical object which records the possible \emph{initial matroids} of $\M$. 

The tropical Grassmannian was first introduced by Speyer and Sturmfels \cite{speyer_sturmfels}. Its connection to the space of phylogenetic trees and the moduli space of rational tropical curves is a celebrated and motivating result in studying these objects. 
In \cite{speyer}, it is demonstrated that points in $\mathbb{R}^{\binom{n}{d}}$ satisfying the tropicalized Pl\"{u}cker relations induce subdivisions of the $(d,n)$-hypersimplex whose cells are matroid polytopes. These points also correspond to tropical linear spaces. 
It has been observed (e.g., \cite{tropicalbook, tropplanes}) that these points also give valuations on the uniform matroid, as in \cite{dress}. 
The set of all valuations on the uniform matroid was dubbed the \emph{Dressian} in~\cite{tropplanes}. The authors of~\cite{tropplanes} introduce a Dressian for each matroid, whose points are valuations on that matroid.

Since then, many questions about Dressians have been studied. Bounds on the dimension of Dressians were given in \cite{js, tropplanes}. Rays of the Dressian have been studied in \cite{js, hjs}. 
Computing Dressians of uniform matroids has also been completed up to $d=3$ and $n=8$ \cite{tropplanes}. Recently, in \cite{martapaper}, the authors have studied the fan structure of Dressians and prove that the Dressian of the sum of two matroids is given by the product of their Dressians.

In this paper, we investigate the nature of Dressians of matroids further. 
Given a matroid $\M$ with valuation $v:\B(\M)\rightarrow \mathbb{R}\cup \{\infty\}$, we define the initial matroid $\M_v$ (as in \cite{dress,tropideals}) to be the matroid with basis set $
\B_v = \{\sigma \in \B\ |\ v(\sigma)\text{ is minimal}\}.
$
This gives a useful restriction on the notion of a weak map which is compatible with matroid valuations.

In Section \ref{sec:basics}, we give an overview of matroids, Dressians, subdivisions of the matroid polytope, and valuated matroids.
In Section \ref{sec:resclass}, we study initial matroids and their polytopes. 
We show that points in the tropical Grassmannian of a matroid over a field $K$ give weight vectors on the matroid polytope which induce regular matroid subdivisions containing cells corresponding to matroids which are also realizable over the field $K$. We explore failures of the converse to this coming from Speyer's thesis \cite{speyerthesis}, namely examples where all cells of a regular matroid subdivision are polytopes of realizable matroids, but the point of the Dressian inducing the subdivision is not contained in the Grassmannian.


In Section~\ref{sec:algorithms}, we turn to the problem of effectively computing Dressians, and give Algorithm~\ref{algo:reduction} which reduces the number of variables and polynomials needed for computing Dressians. An implementation of Algorithm 1 can be found at 
\url{https://github.com/madelinevbrandt/dressians}.
This yields efficiencies which speed up the computations of Dressians of matroids. This is used when the Dressian is contained in a classical linear space; geometrically the equation reduction which occurs in the algorithm corresponds to projecting the Dressian onto this linear space.

In Section~\ref{sec:examples}, we use Algorithm~\ref{algo:reduction} to compute the Dressians of the star $10_3$ configuration, the non-Pappus matroid, the V\'{a}mos and non-V\'{a}mos matroids, the Desargues configuration, and others. In these examples, we illustrate interesting features of these Dressians.


In Section \ref{sec:counterexamples}, we use our computational and theoretical results to obtain counterexamples to two reasonable conjectures. We give an example of a finest subdivision whose cells include non-rigid matroids (Theorem \ref{thm:finestsubs}). We also give an example of matroids $M$, $M'$, and $M''$ such that $M'$ is an initial matroid of $M$, and $M''$ is an initial matroid of $M'$, but $M''$ is not an initial matroid of $M$ (Theorem~\ref{theorem:nontransitive}).

\subsection*{Acknowledgements}

We thank Yue Ren and
Paul G\"{o}rlach for their assistance in computing tropical prevarieties. We thank Alex Fink, Felipe Rinc\'{o}n, and Mariel Supina for several useful conversations. 
We thank an anonymous referee for comments on an earlier version of this paper.
Finally, we thank Bernd Sturmfels for his comments and suggestions. 

 This material is based upon work supported by the National Science Foundation Graduate Research Fellowship Program under Grant No. DGE 1752814. Any opinions,
findings, and conclusions or recommendations expressed in this material are those of the
author and do not necessarily reflect the views of the National Science Foundation.

The second author was supported in part by NSF grants DMS-1855135 and DMS-1854225.

\section{Dressians and tropical Grassmannians of matroids}
\label{sec:basics}

We begin with some notions from tropical geometry and matroid theory.
 Let $K$ be an algebraically closed field with a valuation $\val_K$. Let $I$ be an ideal in the Laurent polynomial ring with $n+1$ variables $K[x_0^{\pm 1}, \ldots, x_n^{\pm 1}]$. The \emph{tropical variety} associated to $I$ is defined as
$
\cap_{f \in I} \trop(V(f)),
$
where the $\trop(V(f))$ are the tropical hypersurfaces corresponding to polynomials $f \in I$ (See \cite[Definition 3.3.1]{tropicalbook}). 
For every ideal $I$ there exists a finite subset $B \subset I$ called a \emph{tropical basis} such that the tropical variety is equal to
$
\cap_{f \in B} \trop(V(f)).
$
Using a tropical basis one can compute the corresponding tropical variety. In many cases, however, it is computationally difficult to find a tropical basis. Given any collection of generators $B'$ for the ideal $I$, we call the set
$
\cap_{f \in B'} \trop(V(f))
$
 a \emph{tropical prevariety.}
The \emph{lineality space} of a tropical (pre)variety $T$ is the largest linear space $L$ such that for any point $w\in T$ and any point $v \in L$, we have that $w + v\in T$.

 A \emph{matroid} of rank $d$ on $n$ elements is a collection $\B \subset \binom{[n]}{d}$ called the \emph{bases} of $\M$ satisfying:
 \begin{enumerate}
 \item[(B0)] $\B$ is nonempty,
 \item[(B1)] Given any $\sigma, \sigma' \in B$ and $e \in \sigma' \backslash \sigma$, there is an element $f\in \sigma$ such that $\sigma \backslash \{f\} \cup \{e\} \in \B$.
 \end{enumerate}
 A matroid $\M$ is called \emph{realizable over $K$} if there exist vectors $v_1, \ldots, v_n \in K^d$ such that the bases of $K^d$ from these vectors are indexed by the bases of $\M$:
 $$
 \B = \left \{\sigma \in \binom{[n]}{d} \ \middle|\ \{v_{\sigma_1}, \ldots, v_{\sigma_d}\}\text{ is a basis of }K^d\right \}.
 $$
 In this case, we write $\M = \M[v_1, \ldots, v_n]$.
 The \emph{uniform matroid} $U_{d,n}$ is the matroid with basis set $\binom{[n]}{d}$.
For more information on matroids, we encourage the reader to consult \cite{oxley} or \cite{white}.

The \emph{Grassmannian} $G(d,n) \subset \mathbb{P}^{\binom{n}{d} -1}$ is the image of $K^{d \times n}$ under the \emph{Pl\"ucker embedding}, which sends a $d \times n$-matrix to the vector of its $d \times d$ minors. The entries of this vector are called the \emph{ Pl\"ucker coordinates} of the matrix. The Grassmannian is a smooth algebraic variety defined by equations called the \emph{Pl\"ucker relations}, which give the relations among the maximal minors of the matrix. Points of this variety correspond to $d$-dimensional linear subspaces of $K^n$.
The open subset $G^0(d,n)$ of the Grassmannian parametrizes subspaces whose Pl\"{u}cker coordinates are all nonzero. 
These points correspond to equivalence classes of matrices where no minor vanishes. In other words, these are matrices which give the uniform matroid of rank $d$ on $[n]$. 

We now recall the definition of the tropical Grassmannian and Dressian of a matroid, as in \cite{tropicalbook}.
Let $\M$ be a matroid of rank $d$ on the set $E = [n]$. For any basis $\sigma$ of $\M$, we introduce a variable $p_\sigma$. Consider the Laurent polynomial ring $K[p_\sigma^{\pm 1}\ |\ \sigma\text{ is a basis of }\M]$ in these variables.
Let $G_\M$ be the collection of polynomials obtained from the three-term Pl\"ucker relations by setting all variables not indexing a basis to zero. 
In other words, we start with the relations $p_{Sij}p_{Skl} - p_{Sik}p_{Sjl} + p_{Sil}p_{Sjk}$ for $S \in \binom{[n]}{d-2}$ and $i$, $j$, $k$, $l$ distinct elements of $[n] \setminus S$ and, in each of these trinomials, we replace $p_{\sigma}$ by $0$ if $\sigma$ is not a basis of $\M$.

Let $I_{\M}$ be the ideal generated by $G_\M$. We call $I_\M$ the \emph{matroid Pl\"{u}cker ideal} of $\M$, and refer to elements of $G_\M$ as \emph{matroid Pl\"{u}cker relations}.

The points of the variety $V(I_\M)$ correspond to realizations of the matroid $\M$ in the following sense. Points in $V(I_\M)$ give equivalence classes of $d \times n$ matrices whose maximal minors vanish exactly when those minors are indexed by a nonbasis of $\M$. We will call $V(I_\M)$ the \emph{matroid Grassmannian of $\M$}.
The variety $V(I_\M)$ is empty if and only if $\M$ is not realizable over $K$. 
Its tropicalization $\Gr_\M=\trop(V(I_{\M}))$ is called the \emph{tropical Grassmannian of $\M$}. 
If the rank of $\M$ is 2, then $G_{\M}$ is a tropical basis for $I_{\M}$ \cite[Chapter 4.4]{tropicalbook}.

\begin{definition} The \emph{Dressian} $\Dr_\M$ of the matroid $\M$ is the tropical prevariety obtained by intersecting the tropical hypersurfaces corresponding to elements of $G_\M$:
$$
\Dr_\M = \bigcap_{f \in G_\M} \trop(V(f)).
$$
\end{definition}

By definition, $\Gr_\M \subseteq \Dr_\M$. Equality holds if and only if $G_\M$ is a tropical basis.

Let $\M_1$ and $\M_2$ be matroids with disjoint ground sets $E_1$ and $E_2$ respectively, and basis sets $\B_1$ and $\B_2$ respectively. The \emph{direct sum} of $\M_1$ and $\M_2$ is the matroid
$
\M_1 \oplus \M_2
$
with ground set $E_1 \cup E_2$ and bases $B_1 \cup B_2$ such that $B_1 \in \B_1$ and $B_2 \in \B_2$. A matroid is \emph{connected} if it cannot be written as the direct sum of other matroids. The number of connected components of a matroid is the number of connected matroids it is a direct sum of. In \cite{martapaper}, the authors show that if $\M_1$ and $\M_2$ are matroids with disjoint ground sets, then 
$
\Dr_{\M_1 \oplus \M_2} = \Dr_{\M_1} \times \Dr_{M_2}.
$
For this reason, we will often assume that our matroids are connected.

The \emph{matroid polytope} $P_\M$ of $\M$ is the convex hull of the indicator vectors of the bases of $\M$:
$$
P_\M = \conv\{e_{\sigma_1} + \cdots + e_{\sigma_d}\ |\ \sigma \in \B\}.
$$
The dimension of $P_\M$ is $n - c$, where $c$ is the number of connected components of $\M$ \cite{eva_bernd}.
\begin{theorem}[\cite{ggms}] 
\label{thm:ggms} 
A polytope $P$ with vertices in $\{0,1\}^{n+1}$ is a matroid polytope if and only if every edge of $P$ is parallel to $e_i - e_j$. 
\end{theorem}

Points in the Dressian of $\M$ have an interesting relationship to the matroid polytope of $\M$. Every vector $w$ in  $\mathbb{R}^{|B|}$ induces a regular subdivision $\Delta_w$ of the polytope $P_\M$.
A subdivision of the matroid polytope $P_\M$ is a \emph{matroid subdivision} if all of its edges are translates of $e_i - e_j$. Equivalently, by Theorem \ref{thm:ggms}, this implies all of the cells of the subdivision are matroid polytopes. 

\begin{proposition}[Lemma 4.4.6, \cite{tropicalbook}] Let $\M$ be a matroid, and let $w \in \mathbb{R}^{|B|}$. Then $w$ lies in the Dressian $\Dr_\M$ if and only if the corresponding regular subdivision $\Delta_w$ of $P_\M$ is a matroid subdivision.
\end{proposition}

All matroids admit the trivial subdivision of their matroid polytope as a regular matroid subdivision, so the Dressian $\Dr_\M$ is nonempty for all matroids $\M$.

We now discuss the valuated matroids of \cite{dress}.
 Let $\M$ be a matroid on $E = \{1, \ldots, n\}$ of rank $d$ and bases $\B$. Let $v: \B \rightarrow \mathbb{R}\cup \{\infty\}$ satisfy the following version of the exchange axiom:
 \begin{enumerate}
 \item[(V0)] for $B_1, B_2 \in \B$ and $e \in B_1\backslash B_2$, there exists an $f \in B_2\backslash B_1$ with $B_1' = (B_1 \backslash \{e\}) \cup \{f\} \in \B$, $B_2' = (B_2 \backslash\{f\}) \cup \{e\} \in \B$, and $v(B_1) + v(B_2) \geq v(B_1') + v(B_2')$.
  \end{enumerate}
   We will call $v$ a \emph{valuation on} $\M$, and the pair $(\M,v)$ is called a \emph{valuated matroid} (See \cite{dress} for details). It is known that valuations on a matroid $\M$ are exactly the points in $\Dr_\M$ \cite{tropicalbook}. Indeed, the above condition asserts exactly that the tropicalized matroid Pl\"{u}cker relations hold.


\section{Initial matroids and their polytopes}
\label{sec:resclass}

Let $\M$ be a rank $d$ matroid on $n$ elements which is realizable over a field $K$ with valuation $\val_K$.  Let $\Gamma$ be the value group, let $R$ be the valuation ring of $K$, let $m$ be its maximal ideal, and let $k$ be its residue field. If $K$ is an algebraically closed field and $\val_K$ is a nontrivial valuation, then by the Fundamental Theorem of Tropical Geometry \cite[Theorem 3.2.3]{tropicalbook} points on $\Gr_\M \cap \Gamma^{|B|}$  are all of the form $(\val_K(p_b))_{b\in \B}$ where $(p_b)_{b \in \B} \in (K^*)^{|B|}$ is a point on the matroid Grassmannian $V(I_\M)$. Possibly by multiplying $(p_b)_{b \in \B}$ by an element of $R$, we may assume that $(p_b)_{b \in \B} \in (R)^{|B|}$ and that some coordinate has valuation 0. Let $M$ be a $d \times n$ matrix realizing $\M$ which we may assume is over $R$. Consider the reduction map $\pi: R \rightarrow k$. Then $\pi(M)$ gives a matroid $\M[\pi(M)]$. In what follows we investigate how this matroid is related to $\M$, and in what way it depends on the choice of element in $\Gr_\M$. First, we expand this notion to nonrealizable matroids.

\begin{definition} Let $\M$ be a matroid with bases $\B$ and let $v \in \Dr_\M$. Then the \emph{initial matroid} $\M_v$ is the matroid whose bases are
$
\B_v = \{\sigma \in \B\ |\ v(\sigma)\text{ is minimal}\}.
$
Given a matroid $\M$, the \emph{initial matroids} of $\M$ are the matroids $\M'$ such that there exists a $v \in \Dr_\M$ with $\M_v = \M'$.
\end{definition} 

\begin{remark}
\label{rem:allones}
If $v$ and $w$ are valuations of a matroid $\M$ such that $v-w \in \mathbb{R}(1,1,\ldots,1)$, then they give the same initial matroid: $\M_v= \M_w$. So, we can consider $\Dr_\M$ and $\Gr_\M$ in the tropical projective space $\mathbb{R}^{|B|} / \mathbb{R} (1,1,\ldots,1)$.
The lineality space of $\Gr_{\M}$ is (usually) larger than $\mathbb{R} (1,1,\ldots,1)$. 
However, points which are equivalent modulo lineality may give different initial matroids. We explore the relationship between such matroids in Proposition~\ref{prop:polytope}.
\end{remark}

%
%

We now give an example to illustrate the ideas and results in the rest of the section.

\begin{example}

Let $\M = U_{2,4}$, the uniform rank 2 matroid on 4 elements; $\B = \{01,02,03,12,13,23\}$. 
We now study the Dressian of $\M$. 
 In this case, $G_\M \subset \mathbb{C}\{\{t\}\}[p_{01}, p_{02}, p_{03}, p_{12}, p_{13}, p_{23}]$ consists of the single equation
 $p_{03}p_{12} - p_{02}p_{13} + p_{01}p_{23}.$
 So, we have that the Dressian $\Dr_{M}$ and the Grassmannian $\Gr_\M$ coincide, and they are both described by
 $$
 \min\{p_{03}+p_{12} , p_{02}+p_{13} , p_{01}+p_{23}\}\text{ is attained twice.}
 $$
  The Dressian is a 5 dimensional fan with a four dimensional lineality space. Let the basis for $\mathbb{R}^{|\B|}$ be given by $\{e_{01},e_{02},e_{03},e_{12},e_{13},e_{23}\}$. 
  The lineality space is given by
 $$
 L = \text{span}\left ((1,1,1,0,0,0),(1,0,0,1,1,0),(0,1,0,1,0,1),(0,0,1,0,1,1)\right ).
$$
The Dressian $\Dr_\M$ has 3 maximal cones, which are the rays generated by
$$
r_{01,23} = (1,0,0,0,0,1)\ \  r_{02,13}=(0,1,0,0,1,0)\ \  r_{03,12}=(0,0,1,1,0,0).
$$
The matroid polytope $P_\M$ is the hypersimplex $\Delta(2,4)$, which is an octahedron. Each of the cones of $\Dr_\M$ corresponds to a subdivision of $P_\M$ into two pyramids. Let us study points in the cell of $\Gr_\M$ containing $r_{01,23}$. The point $r_{01,23}$ induces a subdivision where the two maximal cells are the pyramids which are the convex hulls of
$$
P_{01} = \conv\{e_{01},e_{02},e_{03},e_{12},e_{13}\},\ \ \ P_{23} = \conv\{e_{23},e_{02},e_{03},e_{12},e_{13}\}.
$$
The matroid $\M_{r_{01,23}}$ has bases $\{02,03,13,12\}$. Its matroid polytope is the square face which is shared by the pyramids $P_{01}$ and $P_{23}$. Over $\mathbb{C}\{\{t\}\}$, we can realize $\M$ with the matrix
$$
\begin{bmatrix} 
1 & 1 & 1 & 1 \\
1 + t & 1 + 2 t & t & 2t \\
\end{bmatrix},
$$
and the resulting Pl\"{u}cker vector valuates to $r_{01,23}$.  This matrix reduces to a matrix over $\mathbb{C}$ whose matroid is $\M_{r_{01,23}}$. Alternatively, we can also realize $\M$ with the matrix 
$$
\begin{bmatrix} 
1 & 1 & 1 & 1 \\
1 & 2 & 3 & 3 + t^2
\end{bmatrix}.
$$
The Pl\"{u}cker coordinate of this matrix valuates to 
$$v = (0,0,0,0,0,2)=r_{01,23} - (1,0,0,0,0,-1) \in r_{01,23}+L.$$
The matroid $\M_v$ is the matroid with bases $\{01,02,03,12,13\}$, whose matroid polytope is $P_{01}$. Additionally, the matrix above reduces to a matrix over $\mathbb{C}$ whose matroid is exactly $\M_{v}$.
\end{example}

\begin{lemma}
\label{lem:resclass}
Let $\M$ be a rank $d$ matroid on $n$ elements which is realizable over a field $K$ with nontrivial valuation $\val_K$. Let $R$ be the valuation ring of $K$ and let $m$ be its maximal ideal, and $k$ its residue field, with reduction map $\pi$. Let $v \in \Gr_\M$ so that $\min(v) = 0$ and let $M$ be a matrix over $R$ realizing $\M$ whose Pl\"{u}cker coordinate is $v$. Then the initial matroid $\M_v$ is $\M[\pi(M)]$.
\end{lemma}
\begin{proof} The bases of $\M[\pi(M)]$ are indices $\sigma$ of the Pl\"{u}cker coordinate of $\pi(M)$ which do not vanish. In $M$, the corresponding Pl\"{u}cker coordinates necessarily have valuation $0$, and since this is minimal, they will be bases of $\M_v$. Conversely, all Pl\"{u}cker coordinates of $M$ with valuation 0 index columns of $\pi(M)$ whose Pl\"{u}cker coordinates do not vanish, so we have $\M_v = \M[\pi(M)]$, the matroid of $\pi(M)$.
\end{proof}

This lemma tells us that for realizable matroids, initial matroids are reductions, and vice versa. Now, we turn our attention to how initial matroids sit inside the matroid polytope $P_\M$.

\begin{proposition}
\label{prop:polytope}
Let $\M$ be a matroid with matroid polytope $P_\M$, let $v$ be a valuation on $\M$, let $L$ be the lineality space of the Dressian of $\M$, and let $\Delta_v$ be the matroid subdivision of $P_\M$ induced by $v$. Then,
$$
\Delta_v=\{P_{\M_w} \ |\ w \in v + L \}.
$$
\end{proposition}

\begin{proof} First, we show that $P_{M_v}$ is a cell of $\Delta_v$. To that end, we must show that there is a linear functional $l$ on $\mathbb{R}^{|\B|+1}$ whose last coordinate is positive such that the face of $\conv((e_\sigma,v_\sigma)_{\sigma \in \B})$ minimized by $l$ is the matroid polytope of $\M_v$. We obtain $\M_v$ by taking bases $\sigma$ with $v_\sigma$ minimal; in other words, the linear functional $l = (0,\ldots, 0, 1) \in (\mathbb{R}^{|B|+1})^{\vee}$ works.

Now, let $P$ be a polytope in $\Delta_v$. Then, there is a linear functional $l \in (\mathbb{R}^{|\B|+1})^\vee$ with last coordinate scaled to 1 such that $P =
\conv\left(e_\sigma \ |\ l \cdot (e_{\sigma}, v_\sigma) \text{ is minimal}\right).
$
Since $l$ is linear on the vertices of the matroid polytope $P_\sigma$, 
the restriction $l|_{\mathbb{R}^{|\B|}}$ induces the trivial subdivision on $P_\M$, and is therefore contained in the lineality space of the Dressian. Then, the vector $w = (l \cdot e_\sigma)_{\sigma \in B} + v$ is such that $P_{\M_w} = P$.
\end{proof}

\begin{remark}
If $v$ is a valuation on $\M$, the identity map on the ground set $\M\rightarrow \M_v$ is a weak map (see \cite{weakmap}). There are examples of weak maps which do not arise in this way \cite[Section 3]{dress}.
In Theorem~\ref{theorem:nontransitive}, we will give an example of a weak map between connected matroids which does not arise this way, answering a question from~\cite[Question 1]{martapaper}.
 By \cite[Proposition 4.4]{speyer}, when $\M$ is uniform all weak images are initial matroids.
\end{remark}

\begin{remark}
Initial matroids as in \cite[Definition 4.2.7]{tropicalbook} are a special case of the initial matroids here. Let $\M$ be a rank $d$ matroid on $n$ elements. Given a weight vector $w' \in \mathbb{R}^n$, we can make a weight vector $w \in \mathbb{R}^{|\B|}$ by taking 
$
w_\sigma = -\sum_{i \in \sigma} w'_{i}.
$
Any weight vector $w$ arising in this way is in the lineality space of $\Dr_\M$ and induces a trivial subdivision on $P_{\M}$. The initial matroid $\M_w$ will be the initial matroid corresponding to $w'$ by \cite[Proposition 4.2.10]{tropicalbook}. Among the cells of matroid subdivisions of $P_{\M}$, these initial matroids only correspond to faces of $P_{\M}$, while initial matroids in general give all cells of matroid subdivisions by Proposition~\ref{prop:polytope}.
\end{remark}



The Dressian does not depend on the field over which it is defined. On the other hand, the Grassmannian of a matroid, which is always contained in the Dressian, does depend on the residue characteristic of the field. We now give a result which explains the dependence on the residue characteristic, and gives a criterion to distinguish whether a point in the Dressian is contained in the Grassmannian of a matroid. First, we study an example.

\begin{example}
The non-Fano matroid is the rank 3 matroid on 7 elements with nonbases $\{014,$ $025,$ $036,$ $126,$ $234,$ $456\}$. It is depicted in Figure~\ref{fig:nonfano}.
\begin{figure}[ht]
    \centering
    \includegraphics[height = 1.5 in]{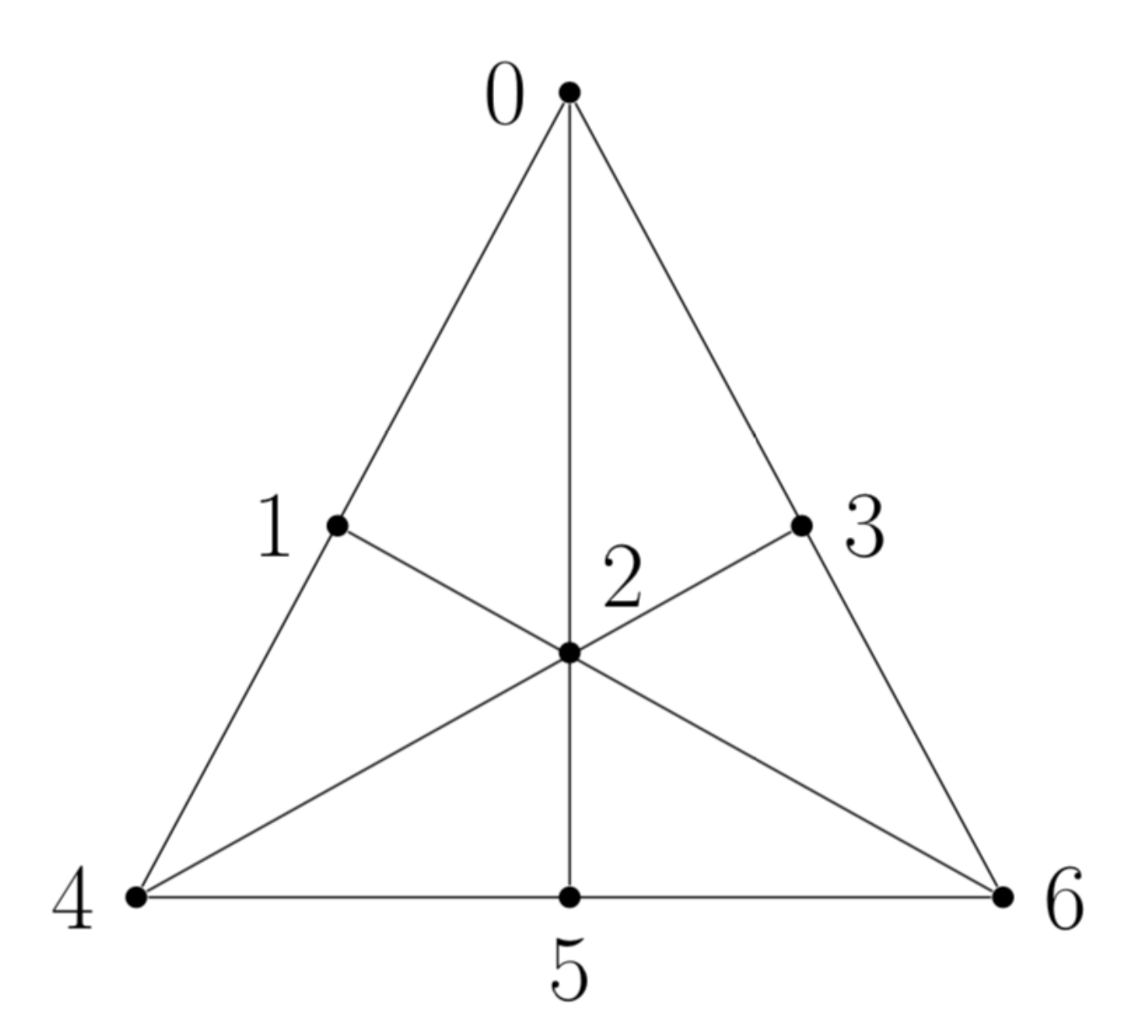}
    \caption{The non-Fano matroid.}
    \label{fig:nonfano}
\end{figure}
Its Dressian has dimension 8 with a 7 dimensional lineality space. Modulo this lineality space, it consists of a single ray. Subdivisions induced by points on the ray contain a cell which is the matroid polytope of the Fano matroid. Over fields which do not have characteristic 2, the Grassmannian consists only of the lineality space. 
Over a field of characteristic 2, the Grassmannian is empty for the following reason: the lineality space corresponds to the trivial subdivision, whose sole facet is not realizable in characteristic 2. Points in the interior of the Dressian are, modulo the lineality space, equivalent to vectors where $p_{135}>0$, $p_{ijk}=0$ for all other bases $(i,j,k)$ of the non-Fano matroid and $p_{ijk} = \infty$ for non-bases. These vectors do not obey the relation from \cite[Proposition 4.5.9]{speyerthesis}.
\end{example}

\begin{proposition}
\label{prop:realizability}
Let $\M$ be a matroid and $K$ be an algebraically closed field with nontrivial valuation $\val_K$ and residue field $k$. Then,
$$
\Gr_\M \subset \{v \in \mathbb{R}^{|\B|}\ |\ \text{all cells of }\Delta_v \text{ are matroid polytopes of matroids which are realizable over }k\} \subset \Dr_{\M}
$$
If $\Gr_\M = \Dr_\M$,
then no regular matroidal subdivision of the matroid polytope $P_\M$ contains a cell which is the matroid polytope of a non-realizable matroid,
and all initial matroids of $\M$ are realizable. Both of the subsets above can be strict.
\end{proposition}

\begin{proof}
Let $v \in \Gr_\M$. By Lemma~\ref{lem:resclass}, the initial matroid $\M_v$ is realizable. By Proposition~\ref{prop:polytope}, $P_{\M_v}$ is a cell of the regular matroid subdivision induced by $v$, and all cells arise in this way.

There are indeed examples of regular matroid subdivisions where all cells correspond to realizable matroids, but a weight vector inducing them is not necessarily contained in the Grassmannian. In his thesis \cite{speyerthesis}, the second author gives two examples of this behavior. 
Example 4.5.6 of \cite{speyerthesis} gives two matroids of rank 3 on 12 elements which are both cells of a regular matroid subdivision of $U(3,12)$ such that the cross ratios of four of the points 5,6,7, and 8 are designed to be two different values. Therefore any weight vector inducing this subdivision cannot be contained in the Grassmannian. 
Example 4.5.8  gives examples of two weight vectors inducing the same subdivision, where one is contained in the Grassmannian and the other is not.
\end{proof}


\section{Linearity and lineality spaces of Dressians} \label{sec:algorithms}

In this section, we study linearity and lineality spaces of Dressians. Since Dressians are tropical prevarieties, they can be computed using software (like \texttt{Gfan} \cite{gfan}). However, these computations become unfeasible for inputs with many polynomials or variables.
In this section, we explain how to reduce these computations to have fewer variables, using linearity. We give an algorithm to carry out this reduction, which we use in the computations in the remainder of the paper. 

 This algorithm is best for computing Dressians of matroids with many non-bases, as it takes advantage of the binomials that the non-bases introduce. The algorithm will not speed up the computation of Dressians of uniform matroids. Fast algorithms for this can be found in \cite{tropplanes,hjs}. Fast algorithms for computing prevarieties in general can be found in \cite{prevarietis}. In \cite[Section 6]{tropplanes}, the authors compute the Dressian of the Pappus matroid, but there is no description of how their computation was performed.

We begin by discussing generalities about reducing the dimension of fans. Let $V$ be a real vector space and let $\Sigma$ be a nonempty fan in $V$. The \emph{linearity space} $\Linearity(\Sigma)$ is the subspace of $V$ spanned by the cones of $\Sigma$ and the \emph{lineality space} $\Lineality(\Sigma)$ is the largest subspace of $V$ which is contained in every maximal cone of $\Sigma$. So $\Sigma$ can be considered as a fan in its linearity space, and that fan is the preimage of a fan in the subquotient space $\Linearity(\Sigma) / \Lineality(\Sigma)$. 
We can reduce dimensions by working in this subquotient. 
As we will now explain, both the linearity and the lineality spaces of $\Dr_M$ and $\Gr_M$ were considered under different names by Dress and Wenzel.

Let $M$ be a matroid with ground set $[n]$ and basis set $\B$. 
Dress and Wenzel~\cite{DW1} introduce an abelian group $\TT^{\B}_M$, defined as follows: the generators of $\TT^{\B}_M$ are called $\epsilon$ and $X(a_1, \ldots, a_k)$, where $(a_1, \ldots, a_k)$ is an ordered basis of $M$. The relations are
\[ \begin{array}{lcll}
\epsilon^2 &=& 1 \\
X(a_{\sigma(1)}, \ldots, a_{\sigma(k)}) &=& X(a_1, a_2, \ldots, a_k) & \mbox{$\sigma$ an even permutation} \\
X(a_{\sigma(1)}, \ldots, a_{\sigma(k)}) &=& \epsilon X(a_1, a_2, \ldots, a_k) & \mbox{$\sigma$ an odd permutation} \\
\end{array}  \]
and that, for matroid elements $a_1$, $a_2$, \ldots, $a_{k-2}$, $b_1$, $b_2$, $c_1$, $c_2$, if the four sets $\{ a_1, a_2, \ldots, b_i, c_j \}$ are bases but $\{ a_1, a_2, \ldots, a_{k-2}, b_1, b_2 \}$ is not a basis, then
\begin{equation} X(a_1, \ldots, a_{k-2}, b_1, c_1) \ X(a_1, \ldots, a_{k-2}, b_2, c_2) = X(a_1, \ldots, a_{k-2}, b_1, c_2)\ X(a_1, \ldots, a_{k-2}, b_2, c_1) . \label{TutteRelation} \end{equation}

Define a map $\Phi : \TT_M^{\B} \longrightarrow \ZZ^n$ by $\Phi(\epsilon) = 0$ and $\Phi(X(a_1, \ldots, a_k)) = \sum_{j=1}^k e_{a_j}$, where $e_1$, $e_2$, \dots, $e_n$ is the standard basis of $\ZZ^n$. The kernel of this map is denoted $\TT_M^0$. 

\begin{theorem}
The linearity space $\Linearity(\Dr_M)$  is naturally contained in $\Hom(\TT_M^{\B}, \RR)$.  If $M$ is realizable, the same is true of $\Linearity(\Gr_M)$. The lineality space $\Lineality(\Dr_M)$ is the image of the map $\Phi^{\ast} : \RR^n \longrightarrow \Hom(\TT_M^{\B}, \RR)$ induced by the map $\Phi :\TT_M^{\B} \longrightarrow \ZZ^n$ ; if $M$ is realizable, the same is true for $\Lineality(\Gr_M)$.
\end{theorem}

Thus, the subquotients $\Linearity(\Dr_M)/\Lineality(\Dr_M)$ and  $\Linearity(\Gr_M)/\Lineality(\Gr_M)$ are subspaces of $\Hom(\TT_M^0, \RR)$.
We caution the reader that we write the group $\TT_M^{\B}$ multiplicatively (following Dress and Wenzel) but write the additive group of $\RR$ additively.

\begin{proof}
We first check that $\Linearity(\Dr_M)$ embeds in $\Hom(\TT_M^{\B}, \RR)$. Since $\epsilon$ is torsion, every group homorphism from $\TT_M^{\B}$ to $\RR$ sends $\epsilon$ to $0$, and we may pass to the quotient of $\TT_M^{\B}$ by $\langle \epsilon \rangle$. 
In this quotient, the element $X(a_1, \ldots, a_k)$ only depends on the unordered basis $\{ a_1, \ldots, a_k \}$, so we can think of $\Hom(\TT_M^{\B}, \RR)$ as a subspace of $\RR^{\B}$. 

Thus, we need to check that, if $( p_{a_1 \cdots a_k} )_{\{ a_1, \ldots, a_k \} \in \B}$ is a point of $\Dr_M$, and if $a_1$, $a_2$, \ldots, $a_{k-2}$, $b_1$, $b_2$, $c_1$, $c_2$ are as in Equation~\ref{TutteRelation}, we must have
\[ p_{a_1 a_2 \cdots a_{k-2} b_1 c_1} + p_{a_1 a_2 \cdots a_{k-2} b_2 c_2}  =  p_{a_1 a_2 \cdots a_{k-2} b_1 c_2} + p_{a_1 a_2 \cdots a_{k-2} b_2 c_1}. \]
Indeed, consider the octahedron whose vertices are indexed by sets of the form $a_1 a_2 \cdots a_{k-2} x y$ where $\{ x, y \} \subset \{ b_1, b_2, c_1, c_2 \}$. The intersection of this octahedron with the matroid polytope $P_M$ must be either a square pyramid or a square, with the vertices $\sum_{j=1}^{k-2} e_{a_j} + e_{b_r} + e_{c_s}$ making up the vertices of the square, so in either case, $p$ must be linear on this square. 	
This completes the check that $\Hom(\TT_M^{\B}, \RR)$ is a subspace of $\Linearity (\Dr_M)$. If $M$ is realizable, then $\Lineality(\Gr_M) \subseteq \Linearity(\Dr_M)$, so the same is true for the Grassmannian.
%

We now show that $\Lineality(\Dr_M) = \Phi^{\ast} \RR^n$. This statement, for the Dressian, appears as Corollary~18 in~\cite{martapaper}, but it appears to us that this reference only checks that $\Phi^{\ast} \RR^n \subseteq \Lineality(\Dr_M)$. We first explain how to check that $\Phi^{\ast} \RR^n \subseteq \Lineality(\Dr_M)$, $\Lineality(\Gr_M)$, and then check the converse. Indeed, let $(t_1, \ldots, t_n) \in \RR^n$ and let $p \in \Dr_M$. Then $(\Phi^{\ast}(t_1, \ldots, t_n) + p)_{a_1 \ldots a_k} = p_{a_1 \ldots a_k} + \sum_{j=1}^k t_{a_j}$. In other words, $\Phi^{\ast}(t_1, \ldots, t_n) + p$ differs from $p$ by a global linear function on $P_M$, and thus $\Phi^{\ast}(t_1, \ldots, t_n) + p$ is also in $\Dr_M$. The same argument applies if we replace $\Dr_M$ with $\Gr_M$.

We now prove the reverse containment: If $\Sigma_1 \subseteq \Sigma_2$ are two nonempty fans, then $\Lineality(\Sigma_1) \subseteq \Lineality(\Sigma_2)$, so it is enough to check the claim for the Dressian. Moreover, $\Lineality(\Sigma) \subseteq \Sigma \cap (- \Sigma)$ for any nonempty fan $\Sigma$, so it is enough to check that $\Dr_M \cap (- \Dr_M) \subseteq \Phi^{\ast} \RR^n$. 
So, suppose that $p$ and $- p$ are both in $\Dr_M$. We want to show that $p$ extends to a linear function on the matroid polytope $P_M$. 
Suppose to the contrary that $p$ does not extend to a linear function on $P_M$. Then, by~\cite[Corollary 16]{martapaper}, $p$ induces a non-trivial subdivision of some octahedron in $P_M$; let the vertices of that octahedron be $Sab$, $Sac$, $Sad$, $Sbc$, $Sbd$ and $Scd$ and choose those labels such that
\[ p_{Sab} + p_{Scd} > p_{Sac} + p_{Sbd} = p_{Sad} + p_{Sbc}. \]
But then
\[ (-p_{Sab}) + (-p_{Scd})  <  (-p_{Sac}) + (-p_{Sbd}) = (-p_{Sad}) + (-p_{Sbc}) \]
so $-p$ is not in $\Dr_M$, a contradiction.
Again, if $M$ is realizable, we have $\Lineality(\Gr_M) \subseteq \Lineality(\Dr_M)$.
\end{proof}

\begin{remark}
We do not know of cases where $\Linearity(\Dr_M)$ is smaller than $\Hom(\TT^{\B}_M,\RR)$.  Restricting our attention to realizable matroids, we also don't know an example where $\Linearity(\Gr_M)$ is other than $\Hom(\TT^{\B}_M,\RR)$.
\end{remark}

We deduce the following corollary, which is due to Dress and Wenzel~\cite[Theorem 5.11]{dress}
\begin{corollary}
If $\TT_M^0$ is a torsion group, then $\Lineality(\Dr_M) = \Linearity(\Dr_M)$, and thus the matroid polytope $P_M$ has no matroidal subdivisions.
\end{corollary}

In particular, Dress and Wenzel~\cite[Theorem 3.6]{DW2} show that, if $\M = \mathbb{P}^n(F)$ for $n \geq 2$ and $F$ a finite field, then $\TT_M^0 \cong F^{\ast}$. Thus they deduce~\cite[Theorem 5.11]{dress}:
\begin{corollary} \label{cor:FinitePlanesRigid}
If $n \geq 2$ and $F$ is a finite field, then the matroid of $\PP^n(F)$ is rigid.
\end{corollary}
 This establishes Conjecture~6.1 from~\cite{martapaper}.

More generally, if the dimension of the vector space $\Hom(\TT_M^0, \RR)$ is low, we expect $\Dr_M$ and $\Gr_M$ to be easy to compute. 
We now explain how our algorithm carries this out in detail.
Given a set of polynomials $G$ whose prevariety has small lineality space relative to its ambient dimension, we will modify the polynomials in $G$ to eliminate the unnecessary ambient dimensions.

\begin{lemma}
Let $K$ be a field with valuation $\val_K$. Let $G \subset K[x^{\pm 1}, y_1^{\pm 1}, \ldots, y_d^{\pm 1}, z_1^{\pm 1}, \ldots, z_k^{\pm 1}]$ be finite with only binomials and trinomials, and suppose $f \in G$ is a binomial in which $x$ has degree 1. Then there is a collection $G'\subset K[y_1^{\pm 1}, \ldots, y_d^{\pm 1}, z_1^{\pm 1}, \ldots, z_k^{\pm 1}]$ such that from the tropical prevariety defined by the $G'$ one can recover the tropical prevariety defined by the $G$. 
\label{lem:linearity}
\end{lemma}

We note that Lemma \ref{lem:linearity} can be viewed as a special case of \cite[Proposition 3.1 (a)]{th09} and \cite[Theorem 3.1]{th07}, but that the computation of Dressians introduces enough new structure that it is useful to state separately.

\begin{proof} Suppose $f = x y_1^{m_1} \cdots y_d^{m_d}+c z_1^{n_1} \cdots z_k^{n_k}$. Then, the tropical hypersurface of $f$ is defined by the equation
$$
x = \val_K(c)+n_1 z_1 + \cdots + n_k z_k - m_1 y_1 - \cdots - m_d y_d .
$$
This equation defines a classical hyperplane, which introduces linearity in to the tropical prevariety defined by the $G$. To obtain $G'$, we substitute $x = -c -z_1^{n_1} \cdots z_k^{n_k} (y_1^{m_1} \cdots y_d^{m_d})^{-1}$ in every equation where $x$ appears in $G$, and the following 3 situations can arise. Let $g \in G$, and denote the substitution map from  $K[x^{\pm 1}, y_1^{\pm 1}, \ldots, y_d^{\pm 1}, z_1^{\pm 1}, \ldots, z_k^{\pm 1}] \rightarrow K[y_1^{\pm 1}, \ldots, y_d^{\pm 1}, z_1^{\pm 1}, \ldots, z_k^{\pm 1}]$ by $\phi$. For any polynomial $g$, let $t(g)$ be the number of terms of $g$.
\begin{enumerate}
    \item If $t(g) = t(\phi(g))$, then we add $\phi(g)$ to $G'$.
    \item \label{case:delete} If $t(\phi(g)) < 2$, then $\trop(g)$ asserts that the minimum of two or more identical linear forms is attained twice, so we do not add $\phi(g)$ to $G'$. See Example \ref{ex:caution}.
    \item
    \label{inequalities}
    If $t(g) = 3$ and $t(\phi(g)) = 2$, then tropically this asserts an inequality. We do not add $\phi(g)$ to $G'$, but we record this inequality.
\end{enumerate}
Then, the tropical prevariety defined by the $G'$ and intersected with any inequalities arising from (\ref{inequalities}) is the projection of $\Gr_\M$ onto the $(y_1,\ldots,y_d,z_1,\ldots,z_k)$ plane. Indeed, for any point $w'$ in the tropical prevariety of $G'$, we can recover a point $w$ in the tropical prevariety of $G$ by adding the coordinate $x = v(c)+n_1 z_1 + \cdots + n_k z_k - m_1 y_1 - \cdots - m_d y_d$. 
\end{proof}

In particular, given distinct matroid elements $a_1$, $a_2$, \ldots, $a_{k-2}$, $b_1$, $b_2$, $c_1$, $c_2$, if the four sets $\{ a_1, a_2, \ldots, b_i, c_j \}$ are bases but $\{ a_1, a_2, \ldots, a_{k-2}, b_1, b_2 \}$ is not a basis, we obtain the binomial relation in Equation~\ref{TutteRelation} to which we may apply this lemma.
We summarize what we have done above in Theorem \ref{thm:reduction}. 

\begin{theorem}
\label{thm:reduction} Let $\M$ be a matroid and let $G_\M$ be the identified generators of its matroid Pl\"{u}cker ideal coming from setting the variables indexing non-bases in the three term Pl\"{u}cker relations to 0. Then Algorithm~\ref{algo:reduction} produces a set of generators $G'$ and linear inequalities $L'$ in fewer variables such that the tropical prevariety (or variety) defined by the polynomials $G'$ intersected with the constraints in $L'$ is isomorphic as a polyhedral complex via a linear map to $\Dr_{\M}$ (or $\Gr(\M)$). 
\end{theorem}

In practice, this is quite a useful trick when computing Dressians of matroids, because it replaces a Laurent polynomial ring in $|\B|$ generators with a Laurent polynomial ring in $\mathrm{Rank}\ \TT_M^{\B}$ generators.
 In the examples in Section \ref{sec:examples}, we will point out the numbers of variables and equations before and after this algorithm is applied. The reduction is implemented in \emph{Mathematica} and can be downloaded at \url{https://github.com/madelinevbrandt/dressians}.

\begin{algorithm}
\caption{Equation reduction for Matroid Pl\"{u}cker Ideals}
\label{algo:reduction}

\begin{algorithmic} 
    \REQUIRE $\M$ a matroid.
    
    \ENSURE $G'$ and $L'$ as in Theorem~\ref{thm:reduction}.
    
    \STATE Compute $G_{\mathcal{M}}$.
    \STATE Let $B$ be the binomials in $G_\M$ having a variable of degree 1.
    \STATE Let $G_{old} = G_\M$.
    \STATE Let $L' = \{\}.$
    
    \WHILE{$|B|>0$}
    \STATE Let $G' = \{\}.$
    \STATE Pick $f \in B$, and a variable $x$ which has degree 1 in $f$. Then we may write $x = m$ for some monomial $m$.
\FOR{$g \in G_{old}$}
\STATE Replace $x$ by $m$ in $g$ to obtain $g'$.
\IF{$t(g) = t(g')$}
\STATE Add $g'$ to $G'.$
\ELSE
\IF{$t(g') = 2$ and $t(g) = 3$}
\STATE{Add the corresponding inequality to $L'$.}
\ENDIF
\ENDIF
\ENDFOR
\STATE Set $B$ to be the binomials in $G'$ having a variable of degree 1.
\STATE Set $G_{old} = G'.$
    \ENDWHILE
\end{algorithmic}

\end{algorithm}

\begin{example}
\label{ex:caution}
We will now give an example that illustrates case $(\ref{case:delete})$ from the proof of Lemma~\ref{lem:linearity}.
Consider the ideal $I \subset K[x^{\pm 1},y^{\pm 1},z^{\pm 1},w^{\pm 1}],$ with generating set $G =\{x y - z w, x y + z w\}$. Then $I$ is the unit ideal, but if we wish to apply Lemma~\ref{lem:linearity} to compute the tropical prevariety, the following happens. We replace $x$ by $z w / y$. The polynomials of $G$, after substitution, are $\{0, 2 z w\}$. This gives an empty prevariety. However, if we instead substitute $x = z + w - y$ in to $\trop(x y + z w)$, we obtain the condition that $\min(z + w, z+w)$ is attained twice for both equations. This is a vacuous constraint, and so any point $(y,z,w)$ can be lifted to a point in the tropical prevariety. This is why we remove polynomials from $G$ which, after substitution, have fewer than two terms. This exact issue appears often when computing Dressians of non-realizable matroids, and never occurs for realizable matroids (because there is no monomial in $I$).
\end{example}

\section{Examples}
\label{sec:examples}

In this section we compute Dressians for interesting matroids analogous to the computation done in \cite[Section 5]{tropplanes} for the Pappus matroid. Every example in this section was computed by first reducing the polynomials via Algorithm \ref{algo:reduction} using the \emph{Mathematica} code provided at \url{https://github.com/madelinevbrandt/dressians}. Then, the output is used to compute the tropical prevariety in \texttt{gfan} \cite{gfan} using the command \texttt{tropicalintersection}.

\subsection{The Star $10_3$}
Consider the rank 3 matroid $\M_\star$ on $\{0,1,\ldots,9\}$ with nonbases given by
$$
\binom{10}{3} \backslash \mathcal{B}_{\M_\star} = \{026,039,058,173,145,169,248,257,368,479\}.
$$
A realization of this matroid is depicted in Figure~\ref{fig:star}. It is a $10_3$ configuration, meaning that it is a configuration of 10 points and 10 lines in the plane such that each point is contained in 3 lines and each line contains 3 points. Up to isomorphism, there are ten such configurations \cite[Table 2.2.7]{grunbaum}. In this table, $\M_\star$ is configuration $(10_3)_3$. The $10_3$ configurations were first determined by Kantor \cite{kantor}. Among them is the Desargues configuration, $(10_3)_1$, which we discuss in Section~\ref{sec:desargues}. The configuration $\M_
\star$ is astral and has orbit type $[2,2]$, meaning that under the action of its symmetry group there are two orbits of points and two orbits of lines, and this is the minimal number that an $n_3$ configuration may have \cite{grunbaum}.

\begin{figure}[ht]
    \centering
    \includegraphics[height=1.8in]{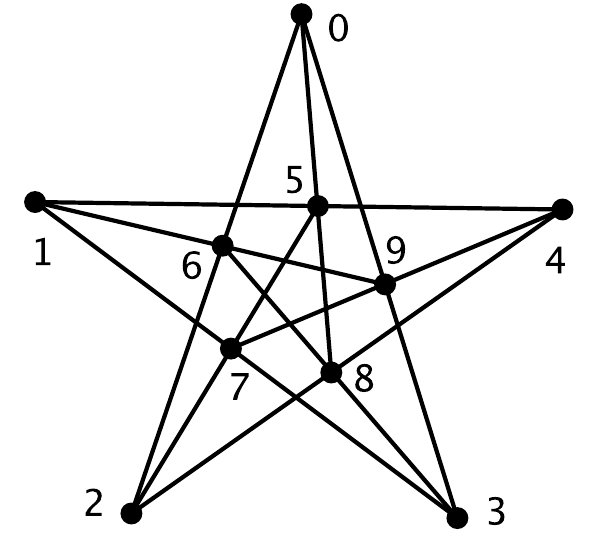}
    \caption{The star matroid $\M_\star$.}
    \label{fig:star}
\end{figure}

\begin{proposition}
Modulo lineality and intersecting with a sphere, the Dressian $\Dr_{\M_\star}$ is a 2 dimensional polyhedral complex with 30 vertices, 65 edges, and 20 triangles. It is depicted in Figure~\ref{fig:star_grass}.
In characteristic 0, the Grassmannian $\Gr_{\M_\star}$ is a graph with 30 vertices and 55 edges. It is depicted in Figure~\ref{fig:star_grass} in the darker color.
\end{proposition}
\begin{proof}
Using Algorithm~\ref{algo:reduction}, we take the generators $G_{\M_\star}$ and make a new generating set whose tropical prevariety will not have linearity. Initially, we are working with 1260 polynomials in 110 variables. After applying Algorithm~\ref{algo:reduction}, we have 73 polynomials in 17 variables. Let $I$ be the ideal generated by these. Using the command \texttt{tropicalintersection} in \texttt{gfan} \cite{gfan}, we obtain the Dressian as claimed.

For each cone $\sigma$ in the Dressian we select a random $w \in \sigma$. Then, we compute $In_w(I)$. If it contains a monomial, then we conclude that the cone is not contained in $\Dr_{\star}$. Doing so demonstrates that no triangle is contained in $\Dr_{\star}$, and neither are the edges which are contained in two triangles. Lastly, to verify that everything else is contained in $\Gr_{\star}$, and to ensure that there was no lower-dimensional cell within a triangle, we check the balancing condition at each ray. We find that the balancing condition holds, and this concludes the proof.
\end{proof}

\begin{figure}[ht]
    \centering
    \includegraphics[height = 3.5 in]{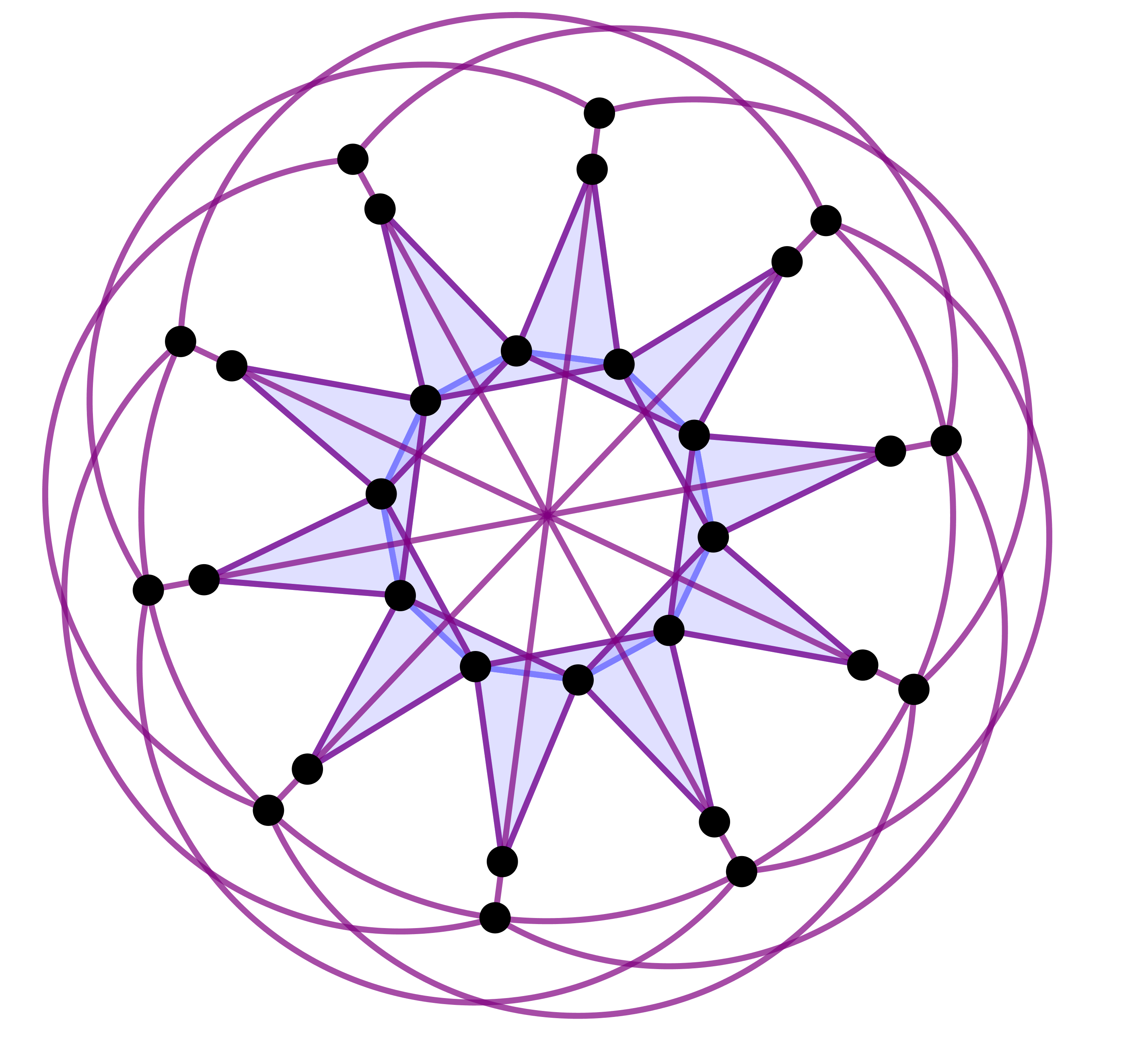}
    \caption{The Dressian and the Grassmannian of $\M_\star$. The Dressian is the full picture, and the Grassmannian is the darkened part.}
    \label{fig:star_grass}
\end{figure}

Let us study the matroid subdivisions arising from rays in the Dressian $\Dr_{\M_\star}$. They come in three tiers, with each tier containing 10 rays. 

Tier I contains the outermost rays in Figure~\ref{fig:star_grass}. Each of these rays induces a subdivision of the matroid polytope with 9 cells, where each cell has the following number of vertices:
$
\{33, 33, 33, 41, 41, 49, 57, 67, 77\}.
$
Tier II contains the middle rays in Figure~\ref{fig:star_grass}. Each of these rays induces a subdivision of the matroid polytope with 6 cells, where each cell has the following number of vertices:
$
\{33, 41, 43, 61, 68, 81\}.
$
Tier III contains the innermost rays in Figure~\ref{fig:star_grass}. Each of these rays induces a subdivision of the matroid polytope with 5 cells, where each cell has the following number of vertices:
$
\{33,33,53,53,96\}.
$
Each of the polytopes with 33 vertices arising in these subdivisions comes from setting six of the points parallel. The resulting matroids are rank 3 matroids on 5 elements with two nonbases, which intersect at a point.






\subsection{Non-Pappus}

In \cite{tropplanes} the authors study the Dressian of the Pappus matroid. They show that as a simplicial complex, it has $f$-vector $f = (18,30,1)$. In \cite[Page 213]{tropicalbook}, Exercise 23, the authors ask for the Dressian of the non-Pappus matroid $\M_{nP}$. This matroid is not realizable over any field, as this would contradict the Pappus Theorem, which says that the points 6,7,8 in Figure~\ref{fig:nonpappus} will always be collinear as long as the other collinearities hold.

\begin{figure}[ht]
    \centering
    \includegraphics[height=2.5 in]{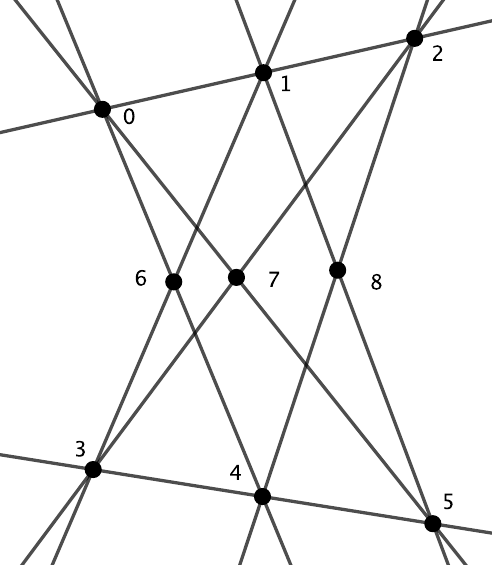}
    \caption{The non-Pappus matroid.}
    \label{fig:nonpappus}
\end{figure}

\begin{proposition}[\cite{tropicalbook}, Chapter 4, Exercise 23]
Modulo lineality and intersecting with a sphere, the Dressian $\Dr_{\M_{nP}}$ is a 3 dimensional polyhedral complex with $f$-vector (19,48,31,1).
The Grassmannian $\Gr_{\M_{nP}}$ is empty.
\end{proposition}
\begin{proof}
Using Algorithm~\ref{algo:reduction}, we take the generators $G_{\M_{nP}}$ and make a new generating set whose tropical prevariety will not have linearity. 
Initially, we are working with 630 polynomials in 76 variables. After applying Algorithm~\ref{algo:reduction}, we have 171 polynomials and 29 variables. Using the command \texttt{tropicalintersection} in \texttt{gfan} \cite{gfan}, we obtain the Dressian as claimed. 
\end{proof}

The projection of this Dressian along the $p_{678}$ axis yields the Dressian of the Pappus matroid. If $f$ is the $f$-vector of the Pappus matroid, we see that $(19,48,31,1) = (f_0+1, f_1+f_0,f_2+f_1,f_2)$.

\subsection{The V\'{a}mos Matroid}

The V\'{a}mos Matroid is a rank 4 matroid on 8 elements which is not realizable over any field. We depict it in Figure~\ref{fig:vamos}. All four-element subsets of the eight elements are bases except $\{0134,\ 0125,\ 2345,\ 3467,\ 2567\}$.

\begin{figure}[ht]
    \centering
    \includegraphics[height = 1.8 in]{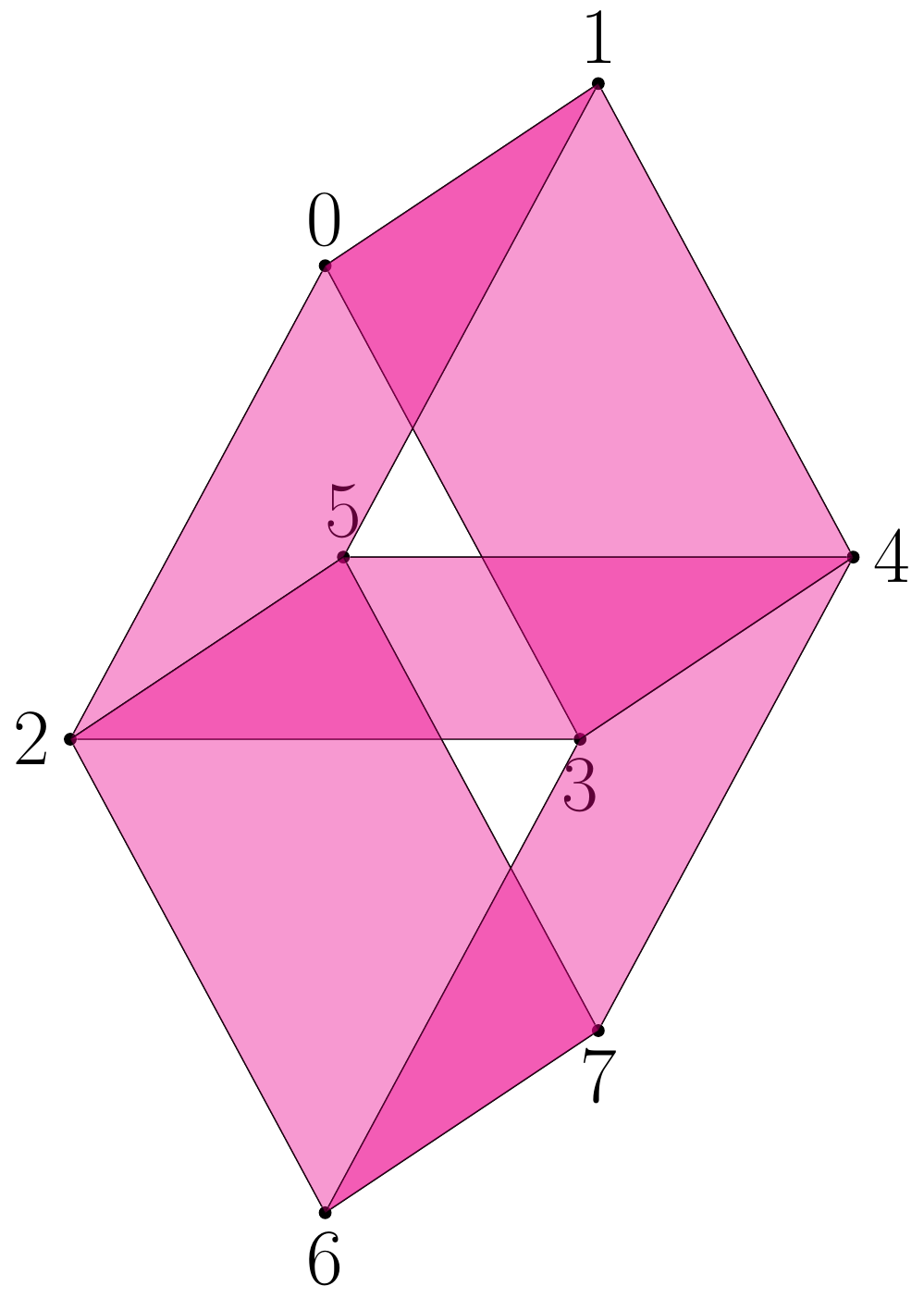}
    \caption{A depiction of the V\'{a}mos Matroid.}
    \label{fig:vamos}
\end{figure}

\begin{proposition}
\label{prop:vamos}
Modulo lineality and intersecting with a sphere, the Dressian of the V\'{a}mos Matroid is an 8 dimensional polyhedral complex with $f$-vector $$
(201, 2014, 6810, 9581, 5425, 896, 72, 18, 2).
$$
\end{proposition}
\begin{proof}
Using Algorithm~\ref{algo:reduction}, we take the generators $G_{\M}$ and make a new generating set whose tropical prevariety will not have linearity. 
Initially, we are working with 420 polynomials in 65 variables. After applying Algorithm~\ref{algo:reduction}, we have 169 polynomials and 33 variables. Using the command \texttt{tropicalintersection} in \texttt{gfan} \cite{gfan}, we obtain the Dressian as claimed. 
\end{proof}

We now study subdivisions of the matroid polytope induced by elements of the two maximal cells. In each case, points from the interior of the cell induce a matroid subdivision which has 9 polytopes with 17 vertices and one polytope with 56 vertices. The large polytopes are the matroid polytopes of the matroids $\M_1, \M_2$ with the following two collections of 14 nonbases:
$$
\overline{\B}_0=\{\underbrace{0134,\ 0125,\ 2345,\ 3467,\  2567}_{\text{from V\'{a}mos}},\ 0167,\ 0246,\ 0356,\ 1247,\ 1357,\ 0237,\ 0457,\ 1236,\ 1456\},
$$
$$
\overline{\B}_1=\{\underbrace{0134,\ 0125,\ 2345,\ 3467,\  2567}_{\text{from V\'{a}mos}},\ 0167,\ 1246,\ 1356,\ 0236,\ 1237,\ 0456,\ 1457,\ 0247,\ 0357\}.
$$
Each of these is the collection of 12 planes in a cube, together with extra nonbases $2345$ and $0167$. The two labellings of the cube are given in Figure~\ref{fig:cubes}.  The matroid polytope of this matroid has no nontrivial matroid subdivisions. This matroid is not realizable over any field, since $\Gr_{\M_1} = \Gr_{\M_2} = \langle 1 \rangle$.
The other 9 polytopes in the subdivision each correspond to one of the above 9 new nonbases. They are each matroids in which all of the elements of the corresponding nonbasis have been parallelized in the V\'{a}mos matroid.

\begin{figure}[ht]
    \centering
    \includegraphics[height = 1.3 in]{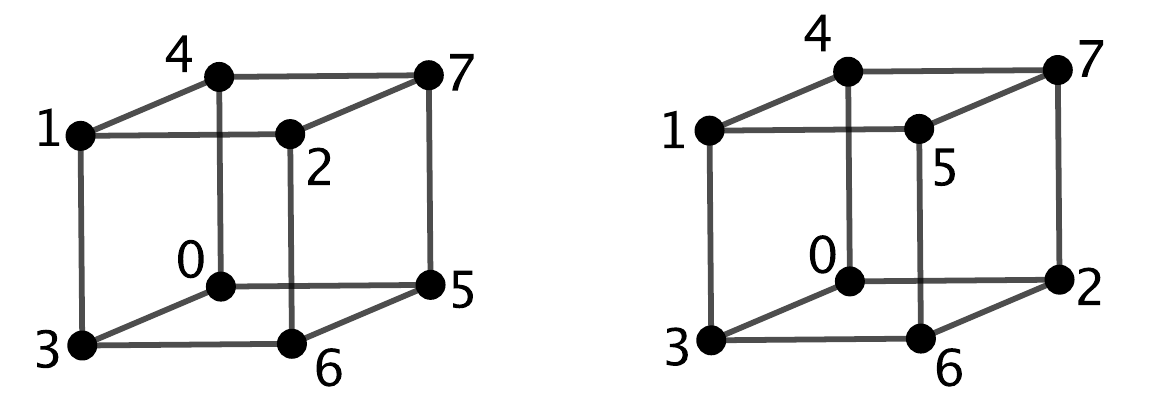}
    \caption{The two cubes whose planes give 12 of the 14 nonbases for the two matroids arising from the V\'{a}mos matroid}
    \label{fig:cubes}
\end{figure}

The non-V\'{a}mos matroid, which has the additional nonbasis $0167$, is realizable. Its Dressian (modulo lineality and intersecting with the sphere) is a 7 dimensional polyhedral complex and has $f$-vector
$$
f = (200, 1814, 4996, 4585, 840, 56, 16, 2).
$$
Like in the case of the Pappus and non-Pappus matroids, we also have here that the Dressian of the non-V\'{a}mos matroid is the projection of the Dressian of the V\'{a}mos matroid along the $p_{0167}$ axis. Indeed, the $f$-vector in Proposition
\ref{prop:vamos}  is
$$(1+f_0,f_0+f_1,f_1+f_2,f_2+f_3,f_3+f_4,f_4+f_5,f_5+f_6,f_6+f_7,f_7).$$

\subsection{Cube}
Consider the matroid $\M_\square$ defined by the cube on the left side of Figure~\ref{fig:cubes}, whose twelve planes define the nonbases. Its Dressian $\Dr_{\M_\square}$, modulo lineality, consists of two points joined by a line segment. The two points each induce a subdivision with two cells, where the matroid corresponding to one cell is one in which one great tetrahedron (i.e., either 0167 or 2345) is collapsed. The points on the segment joining these two points induce subdivisions with three cells, where the largest cell corresponds to the subdivision in which both great tetrahedra have been collapsed.
These matroids are not realizable, so by Proposition~\ref{prop:realizability}, the Grassmannian $\Gr_{\M_\square}$ simply consists of the lineality space.


\subsection{Twisted V\'{a}mos}
\label{sec:twistedvamos}
We now study the Dressian of the rank 4 matroid on 8 elements arising from the polytope depicted in Figure~\ref{fig:maddiehedron} and listed at \cite{polytopia}. The nonbases of this matroid are $\{0123,0145,2345,2567,3467\}$.
 Modulo its lineality space and intersecting with a sphere, this is a five dimensional polyhedral complex with $f$ vector $(120, 1196, 3377, 2985, 397, 8)$.

\begin{figure}
    \centering
    \includegraphics[height = 1.8 in]{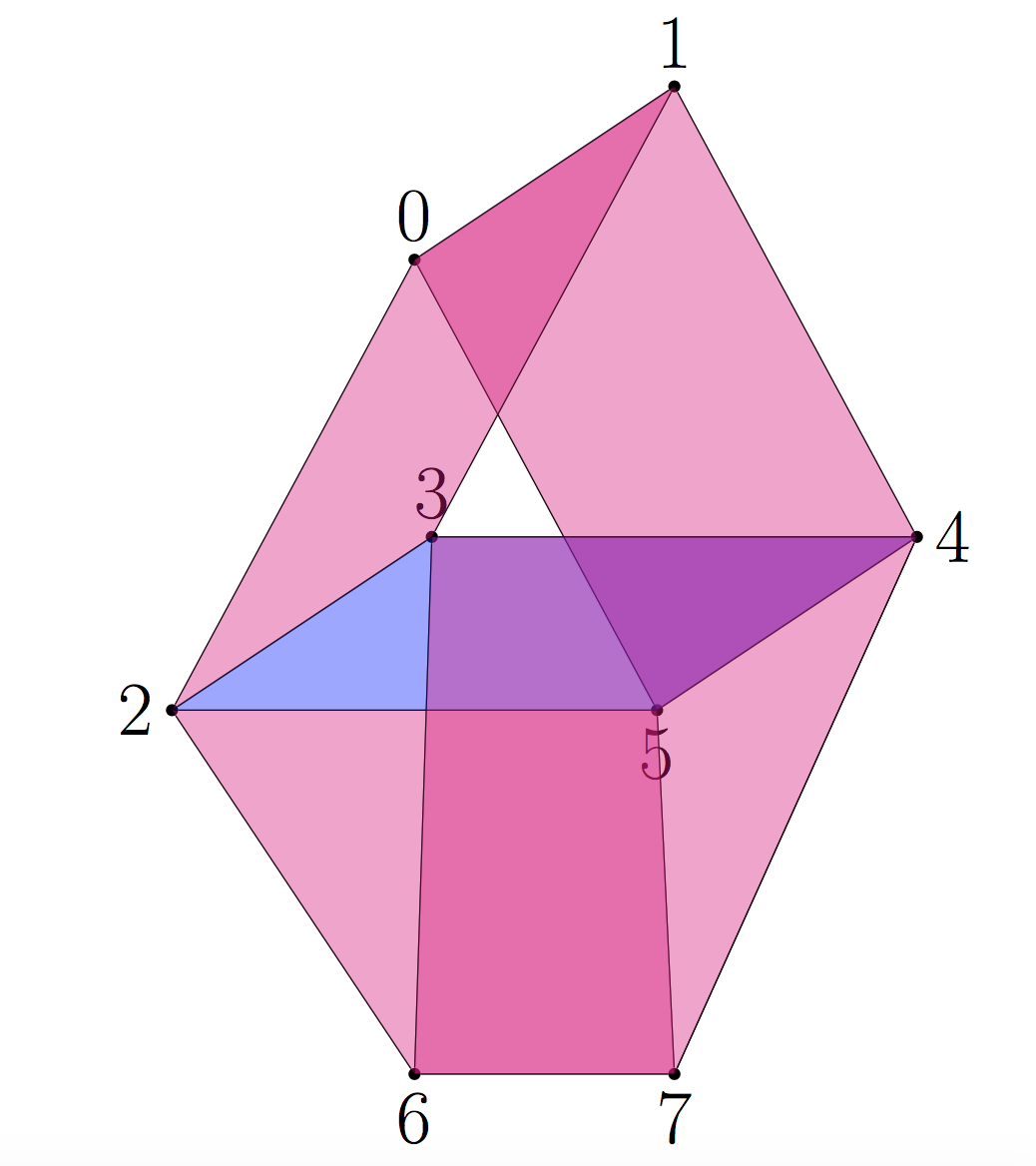}
    \caption{This is a depiction of the twisted V\'{a}mos matroid.}
    \label{fig:maddiehedron}
\end{figure}


\subsection{Desargues}
\label{sec:desargues}

We now study the Desargues configuration. It is named after Gerard Desargues, and the Desargues Theorem proves the existence of this configuration. A depiction of the Desargues configuration is given in Figure~\ref{fig:desargues}. This is another example of a $10_3$ configuration. It is the rank 3 matroid $\M_D$ on $\{0,1,\ldots, 9\}$ with nonbases given by
$$
\binom{10}{3} \backslash \B_{\M_D} = \{027,036,058,135,149,168,234,259,467,789\}.
$$

\begin{proposition}
Modulo lineality and intersecting with a sphere, the Dressian $\Dr_{\M_D}$ is a 3 dimensional polyhedral complex with 70 vertices, 370 edges, 510 two dimensional cells, and 150 three dimensional cells.
\end{proposition}
\begin{proof}
Using Algorithm~\ref{algo:reduction}, we take the generators $G_{\M_{\M_D}}$ and make a new generating set whose tropical prevariety will not have linearity. 
Initially, we are working with 630 polynomials in 74 variables. After applying Algorithm~\ref{algo:reduction}, we have 69 polynomials and 24 variables. Using the command \texttt{tropicalintersection} in \texttt{gfan} \cite{gfan}, we obtain the Dressian as claimed. 
\end{proof}

Of the two dimensional cells, all which are not contained in a larger cell are triangles.
Of the three dimensional cells, 5 are cubes, 30 are pyramids with square bases, and 115 are tetrahedra. The square bases of the pyramids are faces of the cubes. Each pyramid shares a square base with another pyramid. In total, there are 10 vertices which are the tmartapaper of pyramids. The graph on these vertices whose edges correspond to pyramids sharing bases is a Petersen graph. 

\begin{figure}[ht]
    \centering
    \includegraphics[height = 2 in]{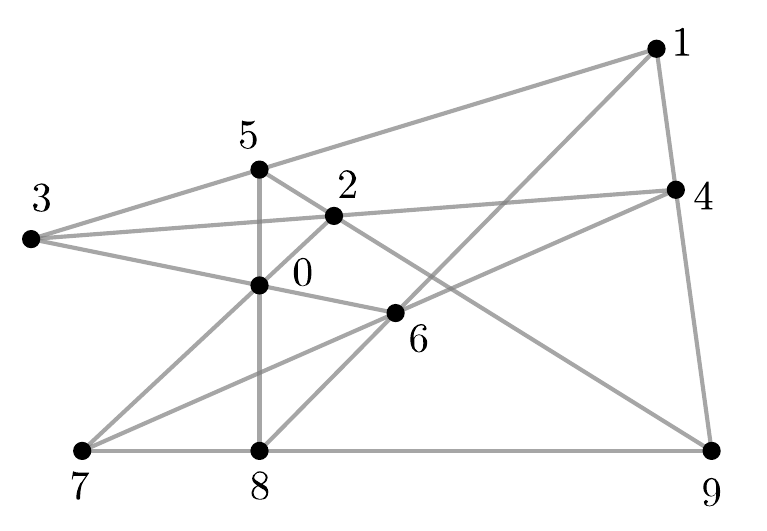}
    \caption{The Desargues Configuration, with the nonbases represented as gray lines.}
    \label{fig:desargues}
\end{figure}


\subsection{A Partial Projective Plane}
\label{sec:proj_plan}

As we explained in Corollary~\ref{cor:FinitePlanesRigid}, Dress and Wenzel showed that the matroid of $\PP^n(F)$ is rigid for any finite field $F$ and any $n \geq 2$. We now consider the matroid of a partial projective plane, which will be used in our counterexamples in Section~\ref{sec:counterexamples}.

A \emph{partial projective plane} is a collection of points $P$, and a collection of subsets of $P$, called \emph{lines}, such that:
\begin{enumerate}
    \item each line contains at least 2 points,
    \item every two points lie on exactly one line, and
    \item every two lines meet in at most one point.
\end{enumerate}

We will construct a partial projective plane $\ppp$,
obtained from $\mathbb{P}^2(\mathbb{F}_3)$. Consider the collection of representatives of equivalence classes in $\mathbb{P}^2(F_3)$ given by
\begin{align*}
\mathbb{P}^2(\mathbb{F}_3) = \{&(1, 0, 0), (1, 0, 1), (1, 0, 2), (1, 1, 0), (1, 1, 1), (1, 1, 2), (1,2, 0),\\& (1, 2, 1), (1, 2, 2), (0, 1, 0), (0, 1, 1), (0, 1, 2), (0, 0, 1)\},
\end{align*}
and label these points $\{0,1,\ldots, 12\}$ respectively. Then the points $\{0,3,6,9\}$ form a line in $\mathbb{P}^2(\mathbb{F}_3)$. Let $\ppp$ be the partial projective plane whose lines are all lines in $\mathbb{P}^2(\mathbb{F}_3)$, with the line $\{0,3,6,9\}$ removed, and add in the six lines $\{\{0,3\},\{6,9\},\{0,6\},\{3,9\},\{0,9\},\{3,6\}\}$. This is a matroid whose bases are $$\mathcal{B}(\ppp) = \mathcal{B}(\mathbb{P}^2(\mathbb{F}_3)) \cup \{\{3,6,9\},\{0,6,9\},\{0,3,9\},\{0,3,6\}\}.$$ We can compute the Dressian of $\ppp$ using Algorithm \ref{algo:reduction}.

\begin{proposition}
Modulo lineality and intersecting with a sphere, the Dressian $\Dr_{\ppp}$ is a 1 dimensional polyhedral complex with 5 vertices and 4 edges, connected as pictured in Figure \ref{fig:p2f3_dress}.
\end{proposition}
\begin{proof}
Using Algorithm~\ref{algo:reduction}, we take the generators $G_{\ppp}$ and make a new generating set whose tropical prevariety will not have linearity. 
Initially, we are working with 6003 polynomials in 238 variables. After applying Algorithm~\ref{algo:reduction}, we have 871 polynomials and 21 variables. Let $I$ be the ideal generated by these polynomials. Using the command \texttt{tropicalintersection} in \texttt{gfan} \cite{gfan}, we obtain the Dressian as claimed. 
\end{proof}

\begin{figure}[h]
    \centering
    \includegraphics[height=1.5 in]{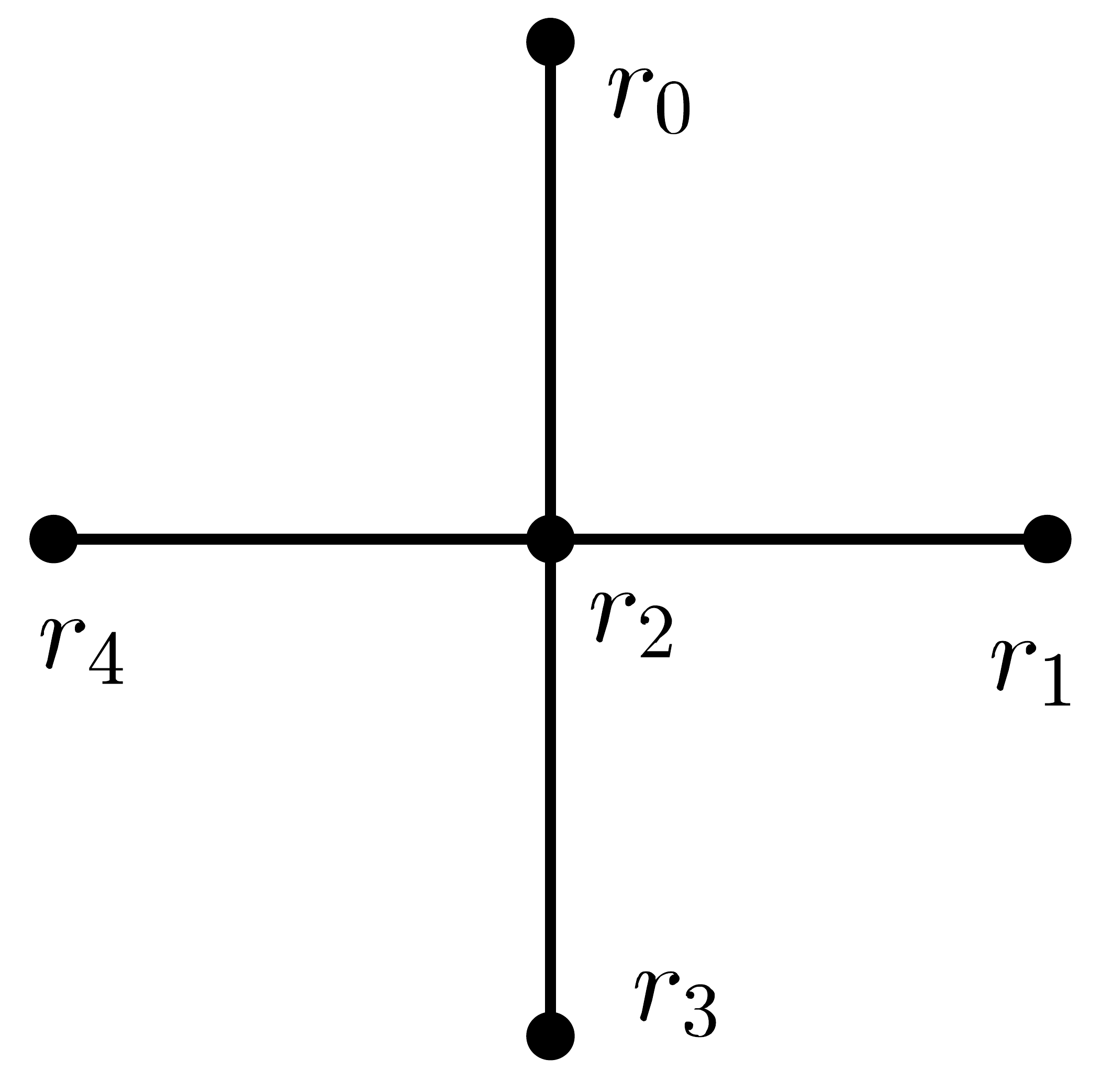}
    \caption{The Dressian of the partial projective plane described in Section \ref{sec:proj_plan}.}
    \label{fig:p2f3_dress}
\end{figure}

We now describe the subdivisions induced on the matroid polytope $P_{\ppp}$ by various points in this Dressian.
We use the following notation in Table \ref{tab:p2f3_subs}.
\begin{enumerate}
    \item The matroid
    $\mathcal{N}_i$ for $i \in \{0,3,6,9\}$ is the partial projective plane obtained from $\mathbb{P}^2(\mathbb{F}_3)$ by removing $i$ from the line $\{0,3,6,9\}$, and adding the lines $\{ij\}$ for $j \in \{0,3,6,9\}\backslash \{i\}$.
    \item  The matroid $\M(i,j,k)$ for distinct $i,j,k \in \{0,3,6,9\}$ is the parallel extension of $\mathcal{U}(3,4)$ where $i,j,k$ each form a parallel class and all other elements are one parallel class.
    \item The matroid $\M(0,3,6,9)$ is the parallel extension of $\mathcal{U}(3,5)$ where $0,3,6,9$ are each in their own parallel class and all other elements are one parallel class.
    \item The matroid $\M(0,3,6,9) - \{ijk\}$ for distinct $i,j,k \in \{0,3,6,9\}$ is the matroid with bases $\mathcal{B}(\M(0,3,6,9))\backslash\{ijk\}$.
\end{enumerate}

\begin{table}[h]
    \centering
    \begin{tabular}{c|c}
        Cell of $Dr_\M$ & Matroids in the subdivision of $P_{\ppp}$ \\
        \hline
        $r_0$ &  $\mathcal{N}_9$, $\M(0,3,6)$ \\
        $r_1$ & $\mathcal{N}_6$, $\M(0,3,9)$\\
        $r_2$ & $\mathbb{P}^2(\mathbb{F}_3)$, $\M(0,3,6,9)$\\
        $r_3$ & $\mathcal{N}_3$, $\M(0,6,9)$\\
        $r_4$ & $\mathcal{N}_0$, $\M(3,6,9)$\\
        $\overline{r_0r_2}$ & $\mathbb{P}^2(\mathbb{F}_3)$, $\M(0,3,6)$, $\M(0,3,6,9) - \{0,3,6\}$\\
        $\overline{r_1r_2}$ & $\mathbb{P}^2(\mathbb{F}_3)$, $\M(0,3,9)$, $\M(0,3,6,9) - \{0,3,9\}$\\
        $\overline{r_2r_3}$ & $\mathbb{P}^2(\mathbb{F}_3)$, $\M(0,6,9)$, $\M(0,3,6,9) - \{0,6,9\}$\\
        $\overline{r_2r_4}$ & $\mathbb{P}^2(\mathbb{F}_3)$, $\M(3,6,9)$, $\M(0,3,6,9) - \{3,6,9\}$\\
    \end{tabular}
    \caption{Subdivisions of $P_\M$ induced by points in $Dr_\M$. 
    }
    \label{tab:p2f3_subs}
\end{table}

Every subdivision of $P_{\M(0,3,6,9)}$ is pulled back from a subdivision of $P_{\mathcal{U}_{3,5}}$. So, its Dressian is the Petersen graph. Even though $\M(0,3,6,9)$ appears in a matroid subdivision of $\M$, not all cells appearing in subdivisions of $\M(0,3,6,9)$ appear in subdivisions of $\M$. See Section \ref{sec:counterexamples} for further discussion.


\section{Counterexamples} \label{sec:counterexamples}

Using our computational tools, we are able to provide counterexamples to two plausible claims about matroid rigidity and initial matroids.
We recall that a matroid $M$ is called \emph{rigid} if $P_M$ has no matroidal subdivisions. 
If all maximal cells in a regular matroid subdivision correspond to rigid matroids, then the subdivision cannot be refined to a finer matroid subdivision.
We will now give a counterexample showing that the converse does not hold -- a matroid subdivision can be nonrefineable and yet contain nonrigid matroids.
This answers Question~2 from~\cite{martapaper}, and fixes an error in a preprint version of this paper~\cite[Theorem B]{paperv1}.

\begin{theorem}
\label{thm:finestsubs}
There are finest matroid subdivisions of matroid polytopes containing maximal cells which are not matroid polytopes of rigid matroids.
\end{theorem}

\begin{proof}
Consider the matroid $\PP^2(\FF_3)$, a matroid of rank $3$ on $13$ elements, and fix a numbering of the elements of $\Delta(3,13)$, so we can think of $P_{\PP^2(\FF_3)}$ as a subset of $\Delta(3,13)$. We will build a regular subdivision of $\Delta(3,13)$ which will have $P_{\PP^2(\FF_3)}$ as a facet. 

We first discuss a subdivision $\mathcal{D}_0$ of $\Delta(3,13)$ which is not finest, but is more symmetrical. For $\{ a,b,c \} \subset [13]$, let $p_{abc}$ be $0$ if $\{ a,b,c \}$ is a basis of $\PP^2(\FF_3)$ and let $p_{abc}$ be $1$ if $\{ a,b,c \}$ is not a basis. This is a tropical Pl\"ucker vector, and the induced subdivision of $\Delta(3,13)$ has $14$ facets, which we will now describe. One of those facets is $P_{\PP^2(\FF_3)}$. The other $13$ facets are indexed by the 13 lines $\{ a,b,c,d \}$ in $\PP^2(\FF_3)$; and we denote them $M(a,b,c,d)$, as in Section~\ref{sec:proj_plan}. The matroid $M(a,b,c,d)$ is a parallel extension of $U_{3,5}$, where each of the elements $a$, $b$, $c$ and $d$ forms a singleton parallelism class, and the other $9$ elements of $\PP^2(\FF_3)$ form the remaining class. 

Let $\mathcal{D}$ now be a finest subdivision of $\Delta(3,13)$ refining this one. Let $\{a,b,c,d \}$ be a line of $\PP^2(\FF_3)$ and let $e$ be another point not on the line. For any $\{ x,y\} \subset \{ a,b,c,d \}$, the triple $\{ x,y,e \}$ is a basis of $\PP^2(\FF_3)$, so these six bases of $\PP^2(\FF_3)$ form an octahedron in $P_{\PP^2(\FF_3)}$. Since $\PP^2(\FF_3)$ is rigid, $P_{\PP^2(\FF_3)}$ must appear, unsubdivided, in $\mathcal{D}$. So the subdivision of $M(a,b,c,d)$ induced by $\mathcal{D}$ leaves this ocathedron unsubdivided as well.

We now describe the subdivisions of $M(a,b,c,d)$. There is a map from $P_{M(a,b,c,d)}$ to $\Delta(3,5)$ induced by the linear transformation sending $e_a\mapsto e_1$, $e_b \mapsto e_2$, $e_c \mapsto e_3$, $e_d \mapsto e_4$ and $e_x \mapsto e_5$ for $x \not\in \{a,b,c,d\}$. Every matroid subdivision of $\M(a,b,c,d)$ is pulled back from a matroid subdivision of $\Delta(3,5)$. The condition that a subdivision does not subdivide the octahedra discussed in the previous paragraph is equivalent to saying that the corresponding subdivision of $\Delta(3,5)$ does not subdivide $e_5 + \Delta(2,4)$. There are $4$ such nontrivial subdivisions of $\Delta(3,5)$, each of which contains one rigid rigid piece and one non-rigid piece. 
The non-rigid piece has a single non-basis, which is a three-element subset of $\{ a,b,c,d \}$. The preimage of this non-rigid piece in $\Delta(3,5)$ is a non-rigid matroid in $M(a,b,c,d)$.
We have shown that any subdivision of $\Delta(3,13)$ refining $\mathcal{D}_0$ has a nonrigid piece (in fact, at least $13$ of them). \end{proof}


We considered the notion above of being an initial matroid, where $M'$ is an initial matroid of $M$ if $P_{M'}$ appears as a face in matroid subdivision of $P_M$. 
It would be natural to guess that this is a transitive relation, but in fact it is not. 
There is a fairly easy counter-example, which was observed in~\cite{martapaper}: let $M = \PP^2(\FF_3)$. Then $P_M$ has octahedral faces, let $P_{M'}$ be one of these. These octahedra can, in turn, be split into two square pyramids; let $P_{M''}$ be one of these. 
Then $M'$ is an initial matroid of $M$, since $P_{M'}$ is a face of $P_M$ by Proposition \ref{prop:polytope}, and $M''$ is an initial matroid of $M'$. However, $M''$ is not an initial matroid of $M$; since $M$ is rigid, the only initial matroids of $M$ correspond to faces of $P_M$, and $P_{M''}$ is not a face of $M$.

The matroid $M'$ in the pervious paragraph is the direct sum of $U_{2,4}$ with a co-loop and eight loops, and is thus disconnected. 
It is natural to ask whether such a counter-example can involve only connected matroids. We now provide such a counter-example, with the aid of the algorithmic tools from this paper. 
This result corrects Proposition~3.2 from an earlier version of this manuscript~\cite{paperv1}. This answers \cite[Question 1]{martapaper}, giving an example of two matroid polytopes $P_{\M''} \subset P_{\M}$ where no matroidal subdivision of $P_\M$ has $P_{\M''}$ as a cell.

\begin{theorem}
\label{theorem:nontransitive}
Let $\M$, $\M'$, and $\M''$ be matroids of rank $d$ on $n$ elements. If $\M'$ is an initial matroid of $\M$ and $\M''$ is an initial matroid of $\M'$, it is not necessarily the case that $\M''$ is an initial matroid of $\M$, even if the matroids are connected. 
\end{theorem}

\begin{proof}
We use a variant of the subdivision from the proof of Theorem~\ref{thm:finestsubs}. Let $\M$ be the matroid $\ppp$ from  Section~\ref{sec:proj_plan} -- in other words, take $\PP^2(\FF_3)$, and declare the four points on a particular line to no longer be collinear, and to have no three of them collinear.
We will number the points of this matroid as in Section~\ref{sec:proj_plan}.

Let $\M'$ be the parallel extension of $U(3,5)$ where $\{0 \}$, $\{ 3 \}$, $\{ 6 \}$ and $\{ 9 \}$ each form singleton parallel classes, and the other $9$ elements of $\{ 0, 1,\ldots, 12 \}$ form a single parallel class. The matroid polytope $P_{\M}$ can be subdivided into two pieces, one of which is $\PP^2(\FF_3)$ and the other of which is $\M'$. Let $\M''$ be the matroid with the same parallel classes as $\M'$, but where $\{ 0,3,6 \}$ are declared to be colinear. It is easy to chop $P_{\M'}$ into two pieces, one of which is $P_{\M''}$. So $\M''$ is an initial matroid of $\M'$ and $\M'$ is an initial matroid of $\M$.

However, examining the list of matroid subdivisions of $P_{\ppp}$ in Section~\ref{sec:proj_plan}, we see that $P_{\M''}$ does not occur as a face in any of them.

We note that the identity map between the ground sets of $\M$ and $\M''$ is a weak map, so this is also an example of a weak map where $\M''$ is not an initial matroid of $\M''$.
\end{proof}

\bibliographystyle{alpha}
\bibliography{sample}

\end{document}